\documentclass[10pt]{article}
\usepackage{amsmath,amssymb,amsthm,latexsym}

\newtheorem{theorem}{Theorem}[section]
\newtheorem{proposition}[theorem]{Proposition}

\newtheorem{corollary}[theorem]{Corollary}
\theoremstyle{definition}
\newtheorem{definition}[theorem]{Definition}

\theoremstyle{remark}
\newtheorem{remark}[theorem]{Remark}

\newcommand{\N}{\mathbb{N}}
\newcommand{\Z}{\mathbb{Z}}

\newcommand{\T}{\ensuremath{\mathbb{T}}}
\newcommand{\Cs}{$\mathrm{C}^*$-al\-ge\-bra}

\newcommand{\cU}{\mathcal{U}}

\DeclareMathOperator{\id}{id}

\begin{document}
\title{Elementary amenable groups are quasidiagonal}
\author{
Narutaka Ozawa \thanks{The first named author was supported by JSPS KAKENHI Grant Number 26400114}\\
RIMS, Kyoto University,\\ 
Sakyo, Kyoto Japan
\and
Mikael R\o rdam \thanks{The second named author was supported  by the Danish National Research Foundation (DNRF) through the Centre for Symmetry and Deformation at University of Copenhagen, and The Danish Council for Independent Research, Natural Sciences.}\\
 University of Copenhagen\\ 
 Copenhagen, Denmark
\and 
Yasuhiko Sato \thanks{The third author was supported by JSPS (the Grant-in-Aid for Research Activity Start-up
25887031) and Department of Mathematical Sciences in University of Copenhagen.}\\
Kyoto University \\
Sakyo, Kyoto, Japan}
\date{}

\maketitle
\begin{abstract} 
We show that the group \Cs{} of any elementary amenable group is quasidiagonal. 
This is an offspring of  recent progress in the classification theory of nuclear \Cs s.
\end{abstract}

\section{Introduction}\label{Sec1}

Rosenberg proved in 1987, \cite{Hadwin}, that if the reduced \Cs{} of a group is \emph{quasidiagonal}, then the group is amenable. 
He suggested that the converse is also true. We confirm this conjecture for \emph{elementary amenable groups}. 

Quasidiagonal \Cs s have been studied since the 1970's. Loosely speaking, a \Cs{} is quasidiagonal if 
it has a faithful approximately block-diagonal representation on a Hilbert space. 
Every quasidiagonal \Cs{} is stably finite. There is no known example of a stably finite nuclear \Cs{} which is not quasidiagonal. Rosenberg's conjecture asserts that there is no such example among group \Cs s. Quasidiagonality plays a central role in the classification program for simple nuclear \Cs s and for our understanding of stably finite \Cs s in general, as for example illustrated in the recent paper, \cite{MS}. We refer the reader to the nice exposition, \cite{Brown:QD}, by N.\ Brown for more information about quasidiagonal \Cs s. 

Rosenberg's conjecture was studied in the recent paper \cite{CarDadEck}, where it is shown that the group \Cs{} of an amenable group is quasidiagonal if and only if the group is MF. This, in turn, allowed the authors to conclude for example that amenable LEF groups have quasidiagonal \Cs s. (LEF stands for locally embeddable into finite groups.) They also proved that the group \Cs{} of an amenable group is not necessarily \emph{strongly quasidiagonal}. For example the lamplighter group, and more generally, a class of wreath products, fail to have this property.

The class of elementary amenable groups is a bootstrap class which is built up from finite and abelian groups by successive elementary operations. All elementary amenable groups are amenable (as the name suggests), but the converse does not hold: groups of intermediate growth discovered by R.\ Grigorchuck, \cite{Grigorchuk}, are amenable but not elementary amenable. Grigorchuck's groups are residually finite (and amenable) and therefore have quasidiagonal \Cs s. The commutator group of the topological full group of a Cantor minimal system was shown by H.\ Matui, \cite{Mat-2006}, to be simple and (sometimes) finitely generated, and by K.\ Juschenko and N.\ Monod, \cite{JusMonod}, to be amenable. No infinite, finitely generated simple group can be elementary amenable. It was observed in \cite{CarDadEck,Kerr:Barcelona} that the topological full groups, being LEF, \cite{GM}, have quasidiagonal \Cs s. The list of amenable groups which are not elementary amenable is long and does not stop here. 

Elementary amenable groups can be described more explicitly by transfinite induction, cf.\ Chou, \cite{Chou}. We use here a related description of elementary amenable groups due to Osin (see Proposition~\ref{Osin}). To show that all elementary amenable groups have quasidiagonal \Cs{}s, one is  faced with the following problem (see also Remark~\ref{remark}): if $G$ is the semi-direct product of a group $H$ by the integers $\Z$, and if $C^*_\lambda(H)$ is quasidiagonal, does it follow that $C^*_\lambda(G)$ is quasidiagonal?  We reformulate this question by considering the formally stronger property, that we call $\mathrm{PQ}$, of a group $G$: that the crossed product of $G$ by the Bernoulli shift on the  CAR algebra $\bigotimes_G M_2$ is quasidiagonal, cf.\ Definition~\ref{def:PQ}. Using classification theory, specifically Theorem~\ref{Classification} from \cite{MS}, as well as a result of H.\ Matui on AF-embeddabilty of crossed products of a simple A\T-algebra of real rank zero by integers, we can show that the class $\mathrm{PQ}$ is closed under extensions by the integers $\Z$ (as well as the other operations as required by Osin's theorem). In this way we obtain our main result, Theorem~\ref{maintheorem}. We show in Proposition~\ref{LEPQ} that countable amenable LEF groups, and in fact all countable amenable groups that are locally embeddable into a PQ group,  belong to the class PQ. This implies that LEF groups and any extension of such a group by an elementary amenable group have quasidiagonal \Cs s thus extending the result from \cite{CarDadEck} mentioned above.

\section{Bernoulli action of UHF-algebras}\label{Sec2}

\noindent 
In order to use classification theory of \Cs s to obtain results about group \Cs s, we associate 
to each countable discrete group $G$ a simple and monotracial 
$\mathrm{C}^*$-alge\-bra $B(G)$  by using the Bernoulli shift. 
Here a \Cs{} is said to be \emph{monotracial} if it has a unique tracial state. 

We recall the definition of the Bernoulli shift. 
Let $A$ be a unital $\mathrm{C}^*$-alge\-bra. For a finite set $X$, 
we consider the $X$-fold tensor product $\bigotimes_X A$. 
Throughout the paper the symbol $\otimes$ means the \emph{minimal tensor product} 
of $\mathrm{C}^*$-alge\-bras. 
For an inclusion $Y\subseteq X$ of finite sets, the $\mathrm{C}^*$-alge\-bra 
$\bigotimes_Y A$ is naturally identified with the $\mathrm{C}^*$-sub\-alge\-bra 
$(\bigotimes_Y A)\otimes(\bigotimes_{X\setminus Y}{\mathbb C}1)$ 
of $\bigotimes_X A$. 
For an infinite set $X$, we define $\bigotimes_X A$ to be the inductive 
limit of $\{ \bigotimes_Y A : Y\subseteq X\mbox{ finite}\}$. Thus it is the 
closed linear span of the elementary tensors
$\bigotimes_{x\in X} a_x$, where $a_x \in A$ for all $x \in X$ and $a_x = 1$ for all but finitely 
many $x \in X$. 
It follows that for any inclusion $Y\subseteq X$ of sets, 
there is a natural embedding $\bigotimes_Y A \subseteq \bigotimes_X A$.

When a group $G$ acts on the index set $X$ by permutations, it gives 
rise to an action of $G$ on $\bigotimes_X A$ by $*$-auto\-mor\-phisms. 
This action is called a (noncommutative) \emph{Bernoulli shift} and will be denoted by $\sigma$. 
In other words, $\sigma$ is defined by 
$\sigma_g(\bigotimes_{x\in X} a_x) = \bigotimes_{x\in X} a_{g^{-1}\cdot x}$ 
for an elementary tensor  $a=\bigotimes_{x\in X} a_x \in \bigotimes_X A$.
We denote by $(\bigotimes_X A)\rtimes_\sigma G$ the corresponding crossed product. 
Throughout the paper the symbol $\rtimes$ means the \emph{reduced crossed product}. 
We omit writing $\sigma$ when there is no fear of confusion. 

Let $M_n$ denote the $\mathrm{C}^*$-alge\-bra of $n\times n$ matrices over $\mathbb C$, 
and let $B$ denote the CAR algebra $\bigotimes_{{\mathbb N}} M_2$. For each countable discrete group $G$, 
we associate the \emph{Bernoulli shift crossed product} $\mathrm{C}^*$-alge\-bra   $B(G) = (\bigotimes_{G} B)\rtimes G$, where $G$ acts on itself by the left translation. 

\begin{proposition}\label{functoriality}
For every countable discrete group $G$, the $\mathrm{C}^*$-alge\-bra $B(G)$ 
is simple and monotracial. If $G$ is amenable, then $B(G)$ is nuclear and belongs to the UCT class. 

Moreover, the functor $G\mapsto B(G)$ satisfies the following functorial properties. 
\begin{enumerate}
\item\label{item:incl}
If $G_1\subseteq G_2$, then  $B(G_1)\subseteq B(G_2)$ naturally. 
\item\label{item:union}
If $G_1\subseteq G_2\subseteq\cdots$ is an increasing sequence, 
then $B(\bigcup G_n)=\overline{\bigcup B(G_n)}$. 
\item\label{item:sdp}
If $H$ acts on $G$ by automorphisms, then the action extends to an action $\alpha$ of $H$ on $B(G)$ such that 
 $B(G)\rtimes_\alpha H\cong B(G\rtimes H)$.
\item\label{item:fisg}
If $G_1\subseteq G_2$ is a finite-index inclusion, then there is a faithful 
embedding of $B(G_2)$ into  $B(G_1)\otimes M_{|G_2/G_1|}$. 
\end{enumerate}
\end{proposition}
\begin{proof}
It was shown in \cite[Theorem 3.1]{Kishimoto:1981} that $A \rtimes G$ is simple if $A$ is simple 
and the action of $G$ on $A$ is outer. Thus it suffices to show that the Bernoulli 
shift is an outer action. Take a nontrivial central sequence $p_n$ of projections in 
$B$ and view them as elements in $\bigotimes_{\{1_G\}} B \subseteq \bigotimes_G B$. 
Then, it is also a central sequence in $\bigotimes_G B$, and 
$\|\sigma_g(p_n) - p_n \|=1$ for all $g \in G \setminus \{1_G\}$ and $n \in {\mathbb N}$ 
(because $\sigma_g(p_n)$ and $p_n$ are distinct but commuting projections). 
It follows that $\sigma_g$ is not inner for any $g \in G \setminus \{1_G\}$.

To prove that $B(G)$ is monotracial, it suffices to show that $\tau(x) = \tau(\Phi(x))$ 
for every tracial state $\tau$ on $B(G)$ and $x \in B(G)$, where $\Phi$ is the canonical 
conditional expectation from $B(G)$ onto $\bigotimes_G B$. For this, it further 
suffices to show that $\tau(au_g) = 0$ for every $g \in G \setminus \{1_G\}$ and 
$a \in \bigotimes_F {\mathcal B}$ with $F\subseteq G$ finite.
Here ${\mathcal B}$ is the dense 
subalgebra $\bigcup_n (\bigotimes_{i=1}^n M_2)$ of the CAR algebra $B$, and $u_g$ denotes the canonical unitary element in $B(G)$ that 
implements the automorphism $\sigma_g$. 
Thus, for any $n \in {\mathbb N}$, one can find a copy $C$ of $M_{2^n}$ 
in $\bigotimes_{\{1_G\}} B$ that commutes with $a$. 
Let $\{p_j\}_{j=1}^{2^n}$ be pairwise orthogonal minimal projections in $C$ such that $\sum_j p_j = 1$. 
Since $\tau(p_j\sigma_g(p_j)) = 2^{-{2n}}$ for all $j$ (observe that $\tau|_{\bigotimes_G B}$ is unique 
and satisfies $\tau(ee') = \tau(e)\tau(e')$ for every $e \in \bigotimes_{\{1_G\}} B$ 
and $e' \in \bigotimes_{\{g\}} B$), one has
\begin{align*}
\left|\tau (a u_g)\right| \leq \sum_{j=1}^{2^n} \left|\tau(p_jau_g p_j) \right| 
= \sum_{j=1}^{2^n} \left|\tau(p_j\sigma_g(p_j) u_ga)\right| 
\le \sum_{j=1}^{2^n} \|a\| \tau\big(p_j\sigma_g(p_j)\big) \le 2^{-n} \|a\|.
\end{align*}
As $n \ge 1$ was arbitrary, we conclude that $\tau(au_g)=0$ as desired.

It is well-known that the crossed product of a nuclear \Cs{} by an amenable group is again nuclear, 
see for example \cite[Theorem 4.2.6]{BroOza:book}. J.-L.\ Tu proved in \cite{Tu} that the UCT holds 
for the \Cs{} of any locally compact amenable group\-oid. (It is also known that the \Cs{} of 
an \'etale groupoid is nuclear if and only if the groupoid is amenable---hence nuclearity implies 
that the UCT holds for each \Cs{} in this class.)
We show that $B(G)$ is the \Cs{} of an amenable \'etale groupoid. 
Let $X=\prod_{G\times\N} \Z/2\Z$ be the compact group on which the subgroup 
$H=\bigoplus_{G\times\N} \Z/2\Z$ (viewed as a discrete group) acts from the left. 
By using the isomorphism $\ell_\infty(\Z/2\Z)\rtimes(\Z/2\Z)\cong M_2$, one sees 
the canonical isomorphism 
$C(X)\rtimes H\cong \bigotimes_{G\times\N} M_2\cong\bigotimes_G B$.
The group $G$ also acts on $G\times\N$ on the first coordinate and hence on $X$ and $H$, too. 
These actions are compatible with the $H$ action on $X$ and give rise to 
an action of the semi-direct product group $H\rtimes G$ on $X$. It is routine to check 
\[
C(X) \rtimes (H\rtimes G) \cong (C(X)\rtimes H)\rtimes G \cong ({\textstyle \bigotimes_G B})\rtimes G = B(G).
\]
Therefore, $B(G)$ is isomorphic to the \Cs{} of the amenable \'etale groupoid 
$X \rtimes (H \rtimes G)$ (see for example \cite[Example 5.6.3]{BroOza:book}). 

Now, we proceed to the proof of functoriality. The assertions (\ref{item:incl}) 
and (\ref{item:union}) are rather obvious. 
The $\mathrm{C}^*$-sub\-alge\-bra of $B(G_2)$, generated by 
$\bigotimes_{G_1}B$ and $\{ u_g : g \in G_1\}$, 
is canonically isomorphic to $B(G_1)$; and 
the $\mathrm{C}^*$-alge\-bra $B(\bigcup G_n)$ is generated by 
$\bigcup_n(\bigotimes_{G_n}B)$ and $\{ u_g : g\in\bigcup G_n\}$.

For (\ref{item:sdp}), let us fix the notation. We write a typical element 
in the semi-direct product group $G \rtimes H$ as $gs$ with 
$g\in G$ and $s\in H$. 
The $H$-action on $G$ is denoted by $\alpha$ 
(which is notationally abusive), i.e., $sgs^{-1}=\alpha_s(g)$. 
We observe that $B\cong\bigotimes_{H} B$ and so 
$B(G)\cong (\bigotimes_{G\rtimes H}B)\rtimes G$ 
where $G$ acts on $G\rtimes H$ from the left. 
To avoid a possible confusion, we denote the latter $\mathrm{C}^*$-alge\-bra 
by $\tilde{B}(G)$. 
The group $H$ also acts on $G\rtimes H$ 
from the left: $s(ht)=\alpha_s(h)st$, and we denote by $\alpha$ again the 
corresponding Bernoulli shift action on $\bigotimes_{G\rtimes H}B$. 
Since, for each $s\in H$, the pair 
$\bigotimes_{G\rtimes H}B\ni b\mapsto\alpha_s(b)$ 
and $G\ni g\mapsto u_{\alpha_s(g)}$ is a covariant representation 
in $\tilde{B}(G)$, the map $\alpha$ extends to an action of $H$ 
on $\tilde{B}(G)$. It is not difficult to see 
\[
\tilde{B}(G)\rtimes_\alpha H 
 = \big(({\textstyle \bigotimes_{G\rtimes H} B})\rtimes G\big)\rtimes H
 \cong ({\textstyle \bigotimes_{G\rtimes H} B})\rtimes(G\rtimes H)
 =B(G\rtimes H).
\]

For (\ref{item:fisg}), we use the isomorphisms $B \cong \bigotimes_{G_2/G_1} B$ and
$B(G_1)\cong (\bigotimes_{G_2}B)\rtimes G_1$. 
Hence it suffices to show when $G_2$ acts on a unital \Cs{} $C$, and $G_1$ is a finite-index subgroup of $G_2$, 
then $C\rtimes G_2\subseteq (C\rtimes G_1)\otimes M_{|G_2/G_1|}$. 
Although this is well-known, we sketch a proof for the convenience of the reader.

Put $n = |G_2/G_1|$. Let $\pi$ be a faithful representation of $C$ on a Hilbert space $H$. Then we get a faithful representation $\pi \times \lambda_{G_2}$ of $C \rtimes G_2$  on the Hilbert space $K=H \otimes \ell_2(G_2)$; and  a faithful representation $\pi \times \lambda_{G_1} $ of $C \rtimes G_1$ on the Hilbert space $K_0 = H \otimes \ell_2(G_1)$. 
Write $G_2$ as a disjoint union $\bigcup_{i=1}^n x_i G_1$, for suitable $x_i \in G_2$, and consider the subspaces $K_i = H \otimes \ell_2(x_iG_1)$, $i = 1, \dots, n$, of $K$. Let $P_i$ denote the orthogonal projection from $K$ onto $K_i$ and let $U_i = (\pi \times \lambda_{G_2})(u_{x_i}) = I_H \otimes  \lambda_{G_2}(x_i)$, where $u_g \in C \rtimes G_2$, $g \in G_2$, are the canonical unitaries that generate the action of $G_2$ on $C$.  
Then $U_i(K_0) = K_i$, and we have a $*$-isomorphism 
$$\Phi \colon B(K) \to B(K_0) \otimes M_n, \qquad T \mapsto \big(U_i^*P_iTP_jU_j\big)_{i,j=1}^n, \quad T \in B(K).$$
For each $c \in C$ and $g \in G_2$, one checks that 
$$U_i^*P_i(\pi \times \lambda_{G_2})(c\, u_g)P_jU_j = \begin{cases} (\pi \times \lambda_{G_1})\big(\alpha_{x_i^{-1}}(c) \, u_{x_i^{-1}gx_j}\big), & \text{if} \: x_i^{-1}gx_j \in G_1\\ 0, & \text{else}\end{cases}.$$
It follows that $\Phi$ maps $(\pi \times \lambda_{G_2})(C \rtimes G_2)$ into $(\pi \times \lambda_{G_1})(C \rtimes G_1) \otimes M_n$, and we thus obtain the desired embedding of $C \rtimes G_2$ into $(C \rtimes G_1) \otimes M_n$. 
\end{proof}

\noindent
By the proof of simplicity of $B(G)$ one can obtain that $\sigma$ is \emph{strongly outer}, cf.\ \cite[Definition~2.5]{MSII}. Thus the Bernoulli shift action of $G$  on $\bigotimes_G B$  has the \emph{weak Rohlin property},  cf.\ \cite[Definition~2.5]{MSII},  whenever $G$ is countable and elementary amenable by  \cite[Theorem 3.6]{MSII}.

\section{Crossed products by integers}
We shall use a bootstrap argument to see that the \Cs{} $B(G)$ is quasidiagonal for every elementary amenable group $G$. To do so, we introduce the following property:

\begin{definition} \label{def:PQ}
Let $\mathrm{PQ}$ be the class of all countable groups $G$ for which the reduced crossed product \Cs{} 
$(\bigotimes_G M_2)\rtimes G$ is quasidiagonal.
\end{definition}

\noindent Observe that PQ $\subseteq$ QD $\subseteq$ AG, where QD is the class of groups $G$ for which $C^*_\lambda(G)$ is quasidiagonal and AG is the class of amenable groups. The former inclusion follows from the fact that $C^*_\lambda(G)$ embeds into $(\bigotimes_G M_2)\rtimes G$, and the latter inclusion is Rosenberg's theorem. We show in this section that PQ contains the class of countable elementary amenable groups. 

\begin{remark} \label{rem:B(G)}
We observe that a countable group $G$ belongs to $\mathrm{PQ}$ if and only if the Bernoulli 
shift crossed product $B(G)$ is quasidiagonal. Indeed, any unital embedding of $M_2$ into the CAR-algebra $B$ induces an embedding of $(\bigotimes_G M_2)\rtimes G$ into $B(G)$, so if the latter is quasidiagonal, then so is the former. 
Assume that $(\bigotimes_G M_2)\rtimes G$ is quasidiagonal and consider the diagonal action $\delta$ of $G$ on $G \times G$. Then
\[
\textstyle  \big(\bigotimes_G (\bigotimes_{G} M_2)\big)\rtimes G 
 \cong (\bigotimes_{G^2} M_2)\rtimes_{\delta} G 
 \hookrightarrow (\bigotimes_{G^2} M_{2})\rtimes G^2 
 \cong \big((\bigotimes_G M_2)\rtimes G\big) \otimes \big((\bigotimes_G M_2)\rtimes G\big)
\]
where the first isomorphism is induced by a suitable set bijection $G \times G \to G^2$, and 
where the inclusion of the second \Cs{}  into the third arises from the diagonal embedding $G\hookrightarrow G^2$. The \Cs{} on the right-hand side is quasidiagonal (being the minimal tensor  product of quasidiagonal \Cs s), and the \Cs{} on the left-hand side is isomorphic to $B(G)$ if $G$ is infinite. If $G$ is finite, then $B(G)$ embeds into $B \otimes M_{|G|}$, cf.\ Proposition \ref{functoriality} (iv). In either case we see that $B(G)$ is quasidiagonal. 
\end{remark}

\noindent
Before proceeding, let us recall some facts about the class $\mathrm{EG}$ of elementary amenable groups. 
By definition, $\mathrm{EG}$ is the smallest class of groups that contains all finite and all abelian groups, 
and which is closed under the following four operations: 
taking subgroups, quotients, direct limits, and extensions. 
Let $\mathrm{B}_0$ denote the union of all finite groups and the infinite cyclic group $\Z$. 
Note that $\mathrm{B}_0$ is closed under taking subgroups and quotients.
It was shown by C. Chou, \cite{Chou}, and refined by D.\ V.\ Osin, \cite[Theorem~2.1]{Osi}, that all groups in 
$\mathrm{EG}$ can be built up from the basis $\mathrm{B}_0$ by transfinite induction just using 
direct limits and extensions. We let $\mathrm{EG}_{\mathrm c}$ denote the class of countable 
elementary amenable groups. 

\begin{proposition}[Chou, \cite{Chou} and Osin, \cite{Osi}] \label{Osin}
$\mathrm{EG}$ is the smallest class of groups which contains the trivial group $\{1\}$ and which is 
closed under taking direct limits and extensions by groups from $\mathrm{B}_0$.
In particular, $\mathrm{EG}_{\mathrm c}$ is the smallest class of groups which contains 
the trivial group $\{1\}$ and which is closed under taking countable direct limits and 
extensions by groups from $\mathrm{B}_0$.
\end{proposition}

\noindent
A class $\mathrm{C}$ of groups is closed under extensions by groups from $\mathrm{B}_0$ 
if for all short exact sequences $1 \to N \to G \to G/N \to 1$, that $N \in \mathrm{C}$ 
and $G/N \in \mathrm{B}_0$ implies $G \in \mathrm{C}$.
Thus, by Proposition \ref{functoriality}, to prove $\mathrm{EG}_{\mathrm c}\subseteq\mathrm{PQ}$, 
it remains to show $B(G)\rtimes_\alpha\Z$ is quasidiagonal for every $G\in\mathrm{PQ}$. 
This follows from the following two results from the classification theory of \Cs{}s.

The first is proved in \cite[Corollary 6.2]{MS} using classification theorems by H.\ Lin--Z.\ Niu, \cite{LN}, 
and W.\ Winter, \cite{Win}. We denote by $\cU$ the universal UHF-algebra. 

\begin{theorem}[\cite{LN}, \cite{MS}] 
\label{Classification} 
Let $A$ be a unital separable simple nuclear monotracial quasidiagonal \Cs{} in the UCT class. 
Then $A \otimes \cU$ is an A\T-algebra of real rank zero. In particular, $A$ is embeddable 
into an AF-algebra.
\end{theorem}

\noindent
The classification theorems mentioned above imply that if $A$ and $B$ are \Cs{}s satisfying the assumptions 
of Theorem~\ref{Classification} as well as the \emph{strict comparison}, 
then $A \cong B$ if and only if $K_0(A) \cong K_0(B)$ 
(as ordered abelian groups with distinguished order units) and $K_1(A) \cong K_1(B)$ (as groups). 
Theorem~\ref{Classification} follows from this because $A \otimes \cU$ has strict comparison, \cite{RorUHF}, 
and one can check that the K-theory of $A \otimes \cU$ agrees with the 
$\mathrm{K}$-theory of a simple A\T-algebra of real rank zero. 
AF-embeddability of such an algebra is a consequence of Elliott's classification result, \cite{Ell}, 
see also \cite[Proposition 4.1]{Ror} for a simpler proof of this fact. 

The second is Matui's theorem, \cite[Theorem 2]{Mat}, about AF-embeddability. We recall that every AF-algebra 
is quasidiagonal, and hence every AF-embeddable \Cs{} is quasidiagonal. It is not known whether there is a 
nuclear quasidiagonal \Cs{} which is not AF-embeddable. 

\begin{theorem}[Matui, \cite{Mat}]\label{matuiat}
Let $A$ be a unital separable simple A\T-algebra of real rank zero. 
Then for any $*$-auto\-mor\-phism $\alpha$, the crossed product $A \rtimes_\alpha \Z$ is embeddable into an AF-algebra. 
In particular, $A\rtimes_\alpha \Z$ is quasidiagonal. 
\end{theorem}

\begin{corollary}\label{zcross}
Let $A$ be a unital separable simple nuclear monotracial quasidiagonal \Cs{} in the UCT class. 
Then for any $*$-auto\-mor\-phism $\alpha$, the crossed product $A \rtimes_\alpha \Z$ is embeddable into an AF-algebra. 
In particular, $A\rtimes_\alpha \Z$ is quasidiagonal. 
\end{corollary}
\begin{proof}
Apply Theorems \ref{Classification} and \ref{matuiat} to 
$(A\otimes\cU)\rtimes_{\alpha\otimes\id}\Z$, which contains $A\rtimes_\alpha\Z$. 
\end{proof}

\begin{remark}
We note that $A \rtimes_\alpha G$ is again monotracial if $A$ is a unital monotracial \Cs, $G$ is a countable amenable group, and 
the action $\alpha$ of $G$ has the weak Rohlin property. This follows  easily from the definition of the weak Rohlin property, cf.\ \cite[Definition 2.5]{MSII}, arguing for example as in  \cite[Remark 2.8]{MS:stronglyouter}.
\end{remark}

\noindent
Here is the main result of this paper.

\begin{theorem}\label{maintheorem} \mbox{}
\begin{enumerate}
\item The class $\mathrm{PQ}$ is closed under the following operations: direct limits, subgroups, and extensions by countable elementary amenable groups.
\item $\mathrm{EG}_{\mathrm c}\subset\mathrm{PQ}$. 
\item $C^*_\lambda(G)$ is AF-embeddable for every group $G$ in $\mathrm{PQ}$. 
\item $C^*_\lambda(G)$ quasidiagonal for all elementary amenable groups $G$.
\end{enumerate}
\end{theorem}

\begin{proof}
(i). The first two claims follow from Proposition \ref{functoriality}. 

Let $\mathrm{QQ}$ denote the class of those countable groups $H$ which satisfy 
the following property: whenever $G$ is a group having a normal subgroup $N$ 
from $\mathrm{PQ}$ such that $G/N\cong H$, the group $G$ belongs to $\mathrm{PQ}$. 
Since the trivial group $\{1\}$ belongs to $\mathrm{PQ}$, one has $\mathrm{QQ}\subseteq\mathrm{PQ}$.
We shall prove that $\mathrm{EG}_{\mathrm c}\subseteq\mathrm{QQ}$. 
By Proposition~\ref{Osin}, it suffices to show that $\mathrm{QQ}$ is closed 
under taking countable direct limits and extensions by groups from $\mathrm{B}_0$. 
Let a surjective homomorphism $q\colon G\to H$ such that $\ker q \in \mathrm{PQ}$ be given. 
First, suppose that $H=\bigcup_n H_n$ is a directed union of $H_n \in \mathrm{QQ}$.
Then, one has $G_n:=q^{-1}(H_n) \in \mathrm{PQ}$ and $G = \bigcup_n G_n$. Hence $G$ belongs to PQ. Next, suppose that $H$ has finite-index subgroup $H_0$ which belongs to $\mathrm{QQ}$. 
Then, $G_0 := q^{-1}(H_0)$ is a finite-index subgroup of $G$ which belongs to $\mathrm{PQ}$.
Hence $G \in \mathrm{PQ}$. 
Lastly, suppose that $H$ has a normal subgroup $H_0$ such that $H_0\in\mathrm{QQ}$ 
and $H/H_0\cong\Z$. Then, since $\Z$ is a free group, $G\cong G_0\rtimes\Z$ for 
$G_0:=q^{-1}(H_0) \in \mathrm{PQ}$. That $G \in \mathrm{PQ}$ follows from 
Proposition~\ref{functoriality}.(\ref{item:sdp}) and Corollary~\ref{zcross}.
This finishes the proof of (i).

(ii) follows from (i) and the fact that the trivial group belongs to PQ. 

(iii). If $G$ belongs to PQ, then $B(G)$ satisfies the hypothesis of Theorem~\ref{Classification}, 
cf.\ Proposition~\ref{functoriality} and Remark~\ref{rem:B(G)}. Hence $B(G)$ is AF-embeddable 
by Theorem~\ref{Classification}, whence $C^*_\lambda(G) \subseteq B(G)$ is also AF-embeddable. 

(iv). It follows from (ii) and (iii) that $C^*_\lambda(G)$ is quasidiagonal for all countable elementary amenable groups. When $G$ is uncountable and elementary amenable, it is a direct limit of countable subgroups,  all of which are elementary amenable. Since quasidiagonality is separably determined, 
the reduced group \Cs{} $C^*_\lambda(G)$ is quasidiagonal. 
\end{proof}
%
%

\noindent
The class $\mathrm{PQ}$ is strictly larger than the class $\mathrm{EG}_{\mathrm c}$ 
of countable elementary amenable groups, as the proposition below shows. 
Recall that a group $G$ is said to be LEF if for every finite subset $F\subset G$ there are 
a finite group $H$ and an injective map $\pi\colon F\to H$ such that 
$\pi(gh)=\pi(g)\pi(h)$ whenever $g,h,gh\in F$. In \cite{CarDadEck,Kerr:Barcelona}, 
it is shown that amenable LEF groups have quasidiagonal reduced group \Cs{}s. 
We adapt this and prove the following. 

\begin{proposition} \label{LEPQ}
Let $G$ be a countable amenable group which satisfies the following property: 
for every finite subset $F\subset G$ there are a group $H$ from $\mathrm{PQ}$ and 
an injective map $\pi\colon F\to H$ such that $\pi(gh)=\pi(g)\pi(h)$ whenever $g,h,gh\in F$.
Then, $G$ belongs to $\mathrm{PQ}$. 

In particular, any countable amenable LEF group belongs to PQ, and more particular, any countable amenable residually finite group,  the Grigorchuk group, and topological 
full groups of Cantor minimal systems belong to $\mathrm{PQ}$.
\end{proposition}

\begin{proof}
Let $G = \bigcup_n F_n$ be a directed union of finite subsets. For every $n$, there are 
a group $H_n$ from $\mathrm{PQ}$ and an injective map $\pi_n\colon F_n\to H_n$ that satisfy the 
above-stated condition. 
We construct an embedding of $(\bigotimes_{G} M_2)\rtimes G$ into 
$\prod_n A_n/ \bigoplus_n A_n$ where $A_n=(\bigotimes_{H_n} M_2)\rtimes H_n$ and 
$\prod_n$ (resp.\ $\bigoplus_n$) denotes the $\ell_\infty$ (resp.\ $c_0$) 
direct sum of \Cs{}s. 
Each $A_n$ is quasidiagonal and therefore an MF algebra (in the sense of \cite{BlaKir}). It therefore follows from \cite[Corollary 3.4.3]{BlaKir} that the separable $\mathrm{C}^*$-algebra $(\bigotimes_G M_2) \rtimes G$ is an MF algebra, hence also an NF algebra (in the sense of \cite{BlaKir}) because it is nuclear, and hence quasidiagonal by \cite[Theorem 5.2.2]{BlaKir}. 

We first look at $\bigotimes_G M_2$. 
Let $a \in \bigcup_n(\bigotimes_{F_n}M_2)$ be given. 
Then $a \in \bigotimes_{E}M_2$ for some finite subset $E\subset G$. 
Take $N\in\N$ such that $E\subset F_N$. Then, for every $n\geq N$, the injective map $\pi_n$ 
induces a canonical $*$-homo\-morphism $\rho_n$ from $\bigotimes_{E}M_2$ 
into $\bigotimes_{H_n} M_2 \subset A_n$. Set $\rho_n(a) = 0$ when $n < N$ and 
define $\rho(a) \in \prod A_n/\bigoplus A_n$ by 
$\rho(a)=[(\rho_n(a))_{n=1}^\infty]$. Then $\rho$ extends to a $*$-homo\-mor\-phism 
from $\bigotimes_G M_2$ into $\prod A_n/\bigoplus A_n$. 
Next, for each $g\in G$, we define $u_g \in \prod A_n/\bigoplus A_n$ by 
$u_g=[(u_{\pi_n(g)})_{n=1}^\infty]$. The values $u_{\pi_n(g)}$ 
where they are not defined do not matter, and $u$ is a unitary representation 
of $G$ into $\prod A_n/\bigoplus A_n$. 
These representations are covariant and since $G$ is amenable, they give rise to 
a $*$-homo\-mor\-phism from the reduced crossed product $(\bigotimes_{G} M_2)\rtimes G$ 
into $\prod_n A_n/ \bigoplus_n A_n$. Since the former is simple (or since the canonical 
tracial states are compatible), it is a faithful embedding. 

The Grigorchuk group and the topological full groups of Cantor minimal systems 
are LEF and so satisfy the assumption of this proposition.
\end{proof}

\begin{remark} \label{remark} Observe that our proof that the class $\mathrm{PQ}$ is closed under extensions by $\Z$ relies on the classification theory for \Cs s. We do not know if the class QD is closed under extensions by $\Z$.  If it is, then one can obtain part (iv) of Theorem \ref{maintheorem} directly from Osin's theorem. 
\end{remark}

\begin{remark} Kerr and Nowak proved in \cite[Theorem 3.5]{KerrNowak} that the reduced crossed product $A \rtimes_\alpha G$ is quasidiagonal whenever $G$ is a countable group in QD, $A$ is a separable nuclear \Cs{} $A$, and  $\alpha$ is a \emph{quasidiagonal} action of $G$ on $A$ (in the sense of \cite[Definition 3.2]{KerrNowak}). It follows that a countable group $G$ in QD belongs to the class PQ if the Bernoulli action of $G$ on $\bigotimes_G M_2$ is quasidiagonal. The latter holds for example when  $G$ is residually finite. 
\end{remark}

\noindent
We conclude our paper remarking that our proofs imply the following corollary that may be of independent interest:

\begin{corollary}
For any  group $G$ in PQ there is an action $\alpha$ of $G$ on the universal UHF-algebra $\mathcal{U}$ such that $\mathcal{U}\rtimes_{\alpha} G$ is a simple A\T-algebra of real rank zero.
\end{corollary}
\begin{proof}
It follows from Proposition~\ref{functoriality}, Theorem~\ref{maintheorem}, and Theorem~\ref{Classification} that 
$((\bigotimes_G B) \otimes\cU)\rtimes_{\sigma\otimes\id}G$ is a simple A\T-algebra of real rank zero. 
We note that $(\bigotimes_G B) \otimes\cU \cong \cU$. 
\end{proof}

\noindent Narutaka Ozawa \\
RIMS, Kyoto University\\
Sakyo-ku, Kyoto 606-8502\\
Japan\\
narutaka@kurims.kyoto-u.ac.jp\\

\noindent Mikael R\o rdam \\
Department of Mathematical Sciences\\
University of Copenhagen\\ 
Universitetsparken 5, DK-2100, Copenhagen \O\\
Denmark \\
rordam@math.ku.dk\\

\noindent Yasuhiko Sato \\
Graduate School of Science \\
Kyoto University \\
Sakyo-ku, Kyoto 606-8502\\ 
Japan \\
ysato@math.kyoto-u.ac.jp

\end{document}